\documentclass[11pt]{amsart}
\usepackage{extsizes}
\usepackage{fullpage}
\usepackage{amsfonts,color}
\usepackage{amssymb}
\numberwithin{equation}{section}
\newtheorem{thm}[equation]{Theorem}

\newtheorem{lemma}[equation]{Lemma}
\newtheorem{prop}[equation]{Proposition}

\newtheorem{remark}[equation]{Remark}

\newcommand{\authorfootnotes}{\renewcommand\thefootnote{\@fnsymbol\c@footnote}}%
\makeatother

\begin{document}

\def\q{\mathfrak{q}}
\def\p{\mathfrak{p}}
\def\l{\mathfrak{l}}
\def\u{\mathfrak{u}}
\def\a{\mathfrak{a}}
\def\b{\mathfrak{b}}
\def\m{\mathfrak{m}}
\def\n{\mathfrak{n}}
\def\c{\mathfrak{c}}
\def\d{\mathfrak{d}}
\def\e{\mathfrak{e}}
\def\k{\mathfrak{k}}
\def\z{\mathfrak{z}}
\def\h{{\mathfrak h}}
\def\gl{\mathfrak{gl}}
\def\sl{\mathfrak{sl}}
\def\t{\mathfrak{t}}

\def\Ext{{\rm Ext}}
\def\Hom{{\rm Hom}}
\def\Ind{{\rm Ind}}

\def\res{\mathop{Res}}

\def\GL{{\rm GL}}
\def\SL{{\rm SL}}
\def\SO{{\rm SO}}
\def\O{{\rm O}}

\def\R{\mathbb{R}}
\def\C{\mathbb{C}}
\def\Z{\mathbb{Z}}
\def\Q{\mathbb{Q}}
\def\A{\mathbb{A}}

\def\w{\wedge}

\def\Cat{\mathcal{C}}
\def\HC{{\rm HC}}
\def\HCat{\Cat^\HC}
\def\proj{{\rm proj}}

\def\to{\rightarrow}
\def\To{\longrightarrow}

\def\1{1\!\!1}
\def\dim{{\rm dim}}

\def\th{^{\rm th}}
\def\isom{\approx}

\def\CE{\mathcal{C}\mathcal{E}}
\def\E{\mathcal{E}}

\def\dis{\displaystyle}
\def\f{{\bf f}}                 
\def\g{{\bf g}}
\def\T{{\rm T}}              
\def\omegatil{\tilde{\omega}}  
\def\H{\mathcal{H}}            
\def\W{W^{\circ}}           
\def\Whit{\mathcal{W}}      
\def\ringO{\mathcal{O}}     
\def\S{\mathcal{S}}      
\def\M{\mathcal{M}}      
\def\K{{\rm K}}          
\def\h{\mathfrak{h}} 
\def\N{\mathfrak{N}}    
\def\norm{{\rm N}}       
\def\trace{{\rm Tr}} 
\def\ctilde{\tilde{C}}

\title{Non-vanishing of Rankin-Selberg Convolutions for Hilbert Modular Forms}

\author{Alia Hamieh}
\address[Alia Hamieh]{University of Northern British Columbia, Department of Mathematics and Statistics, Prince George, BC V2N 4Z9 Canada}
\email{alia.hamieh@unbc.ca}

\author{Naomi Tanabe}
\address[Naomi Tanabe]{Bowdoin College, Department of Mathematics, Brunswick, ME 04011 USA}
\email{ntanabe@bowdoin.edu}
\thanks{Research of first author was partially supported by PIMS Postdoctoral Fellowship at the University of Lethbridge}


\date{\today}

\begin{abstract}
In this paper, we study the non-vanishing of the central values of the Rankin-Selberg $L$-function of two
ad\`elic Hilbert primitive forms $\f$ and $\g$, both of which have varying weight parameter $k$. We prove that, for sufficiently large $k$, there are at least $\frac{k}{(\log k)^{c}}$ ad\`elic Hilbert primitive forms $\f$ of weight $k$ for which $L(\frac12, \f\otimes\g)$ are nonzero.
 \end{abstract}

\maketitle

\section{Introduction}

Chowla conjectured in \cite{chowla} that $L(\frac12,\chi)$ is non-zero for all  primitive quadratic characters $\chi$. Indeed, it is widely believed that such a statement should hold for all primitive Dirichlet characters.  Chowla's conjecture remains open to this day, but there have been significant advances towards its resolution. On this front, a breakthrough was achieved by Soundararajan who showed in \cite{sound} that a remarkably high proportion of quadratic Dirichlet $L$-functions do not vanish at the critical point $s=\frac12$.  Likewise, it is also believed that the central values of modular $L$-functions are non-vanishing unless there is a trivial (e.g. sign in the functional equation) or an arithmetical reason for these values to vanish. 
The purpose of this paper is to study the non-vanishing of the family of central values of the Rankin-Selberg $L$-functions associated with two ad\`elic Hilbert modular forms both of which have varying weight parameter $k$. More precisely, it is our aim to prove the following theorem. 
\begin{thm}\label{thm:main}
Let $F$ be a totally real number field such that  the Dedekind zeta function $\zeta_F$ has no Landau-Siegel zero.
Let $\g$ be an ad\`elic Hilbert modular form in $\Pi_k(\n)$, the (finite) set of all primitive forms of weight $k$ and level $\n$ over $F$. Then there exists an absolute constant $c>1$ such that
\[ \#\left\{\f\in\Pi_k(\ringO_{F})\,:\, L\left(\frac12,\f\otimes\g\right)\neq 0\right\}\gg \frac{k}{\log^c k},\qquad \text{as}\;\;k\to\infty.\]
\end{thm}

Such a result is obtained by establishing asymptotics for certain twisted first and second moments. The classical approach for estimating the second moment is very complicated in our setting. To surmount this difficulty, we apply a short and simple alternative by using the Rankin-Selberg unfolding method as elegantly employed by Blomer \cite{blomer}. Non-vanishing results that are similar in strength to our work  are known for
$L$-functions $L(\frac12, f)$ of classical modular forms by the work of Duke \cite{duke} and Lau-Tsang \cite{lau-tsang}. However, for Rankin-Selberg $L$-functions, the best known weight aspect result prior to our work is a lower bound of order $k^{1-\epsilon}$ due to Liu and Masri \cite{liu-masri} which is also based in parts on \cite{blomer}. 

We mention here that another family worth studying is one in which the Hilbert modular forms involved have different weights, one of which is fixed and the other is varying. More precisely, let $\g\in\Pi_{l}(\n)$ for some fixed $l\in2\mathbb{N}^{n}$ and $\n\subset\ringO_{F}$.  By obtaining an asymptotic formula for the harmonic sum of the values $L(\frac12,\f\otimes\g)$ as $\f$ varies in $\Pi_k(\ringO_{F})$ (which the authors showed in \cite[Proposition 2.4]{2016RAMA}) and using the subconvexity bound $L(\frac12,\f\otimes \g)\ll k^{\delta+\epsilon}$ in \cite{mv}, one could show that $\#\left\{\f\in\Pi_k(\ringO_{F})\,:\, L\left(\frac12,\f\otimes\g\right)\neq 0\right\}\gg k^{1-\delta-\epsilon}$. However, establishing a sharper bound will undoubtedly require more elaborate techniques.


\section{Notations and Preliminaries}\label{sec:background}

Throughout the paper, we fix a totally real number field $F$ of degree $n$ over $\mathbb{Q}$, 
and we impose the condition that the Dedekind zeta function $\zeta_F$ has no Landau-Siegel zero. This assumption is used in the proof of Lemma \ref{lem:omega}. It is worth mentioning here that, by the work of Stark \cite{stark}, we know that
every Galois field with odd degree over $\mathbb{Q}$  satisfies this condition.

Once and for all, we fix an order of the real embeddings, $\sigma_j$, of $F$, say $\sigma:=(\sigma_1,\dots,\sigma_n)$. As such, we can identify an element $x$ in $F$ with the $n$-tuple $(x_{1},\dots,x_{n})$ in $\mathbb{R}^{n}$ where $x_j=\sigma_j(x)$. This tuple may be, again, denoted by $x$ when no confusion arises. 
We say $x$ is totally positive and write $x\gg0$ if $x_{j}>0$ for all $j$, and for any subset $X\subset F$, we put $X^{+}=\{x\in X: x\gg0\}$.

We denote the narrow class group of $F$ by $Cl^{+}(F)$ and its cardinality by $h$. We let  $\{\t_{1},\t_{2},\dots,\t_{h}\}$ be a fixed choice of representatives of the narrow ideal classes in $Cl^{+}(F)$. We write $\a\sim\b$ when fractional ideals $\a$ and $\b$ belong to the same narrow ideal class, in which case we have $\a=\xi\b$ for some $\xi$ in $F^{+}$. The symbol $[\a\b^{-1}]$ is used to refer to this element $\xi$ which is unique up to multiplication by totally positive units in $\ringO_{F}$.


To simplify the exposition of this paper, we frequently use multi-index notation as follows: For given $n$-tuples ${x}$ and ${z}$ and a scalar $a$, we set 
$${x}^{{z}}=\prod_{j=1}^{n}x_{j}^{z_{j}},\quad x^a=\prod_{j=1}^n x_j^a, \quad  \text{and}\quad a^{{z}}=a^{\sum_{j=1}^{n}z_{j}}. $$
Such multi-index notation will also be employed to denote certain products of the gamma function and the $J$-Bessel function.

Let $k=(k_1,\dots, k_n)\in 2\mathbb{N}^{n}$, and let $\n$ be an integral ideal in $F$. We denote by $S_{k}(\n)$ the space of ad\`elic Hilbert cusp forms of weight $k$, level $\n$, and with the trivial character. The Fourier coefficient of an ad\`elic Hilbert cusp form $\f$ at an integral ideal $\m\subset\ringO_F$ is denoted by $C_\f(\m)$, after suitable normalization. We say $\f$ is normalized if $C_{\f}(\ringO_{F})=1$.

As it is well-known, an ad\`elic Hilbert cusp form $\f$ in $S_{k}(\n)$ can be viewed as an $h$-tuple $(f_{1},\dots,f_{h})$ of classical Hilbert cusp forms $f_{\nu}$ of weight $k$ with respect to the congruence subgroup
\[\Gamma_\nu(\n)=\left\{\left[\begin{array}{cc} a&b\\c&d\end{array}\right]\in\GL_2^+(F)\,:\, a, d\in\ringO_F, b\in\t_{\nu}\d_F^{-1},c\in\n\t_{\nu}^{-1}\d_F,ad-bc\in\ringO_F^{\times +}\right\},\]
where $\d_F$ is the different ideal of $F$.

The space of ad\`elic cusp forms can be decomposed as $S_{k}(\n)=S_{k}^{\mathrm{old}}(\n)\oplus S_{k}^{\mathrm{new}}(\n)$ where $S_{k}^{\mathrm{old}}(\n)$ is the subspace of cusp forms that come from lower levels, and the new space $S_{k}^{\mathrm{new}}(\n)$ is the orthogonal complement of $S_{k}^{\mathrm{old}}(\n)$ in $S_{k}(\n)$ with respect to the Petersson inner product defined as
\begin{equation}\label{eqn:petersson}
\left<\f,\g\right>_{\n}=\sum_{\nu=1}^h\left<f_\nu,g_\nu\right>_{\n}=\sum_{\nu=1}^h\frac{1}{\mu(\Gamma_\nu(\n)\backslash\h^n)}\int_{\Gamma_\nu(\n)\backslash\h^n}
\overline{f_\nu(z)}g_\nu(z)y^kd\mu(z),
\end{equation}
where $d\mu(z)=\prod_{j=1}^ny_j^{-2}dx_{j}dy_{j}$.

A Hilbert cusp form $\f$ in $S_{k}(\n)$ is said to be primitive if it is a normalized common Hecke eigenfunction in the new space. We denote by $\Pi_{k}(\n)$ the (finite) set of all primitive forms of weight $k$ and level $\n$. If $\f$ is in $\Pi_{k}(\n)$, it follows from the work of Shimura \cite{shimura-duke} that $C_{\f}(\m)$ is equal to the Hecke eigenvalue for the Hecke operator $T_{\m}$ for all $\m\subset\ringO_{F}$. Moreover, since $\f$ is with the trivial character, the coefficients $C_{\f}(\m)$ are known to be real for all $\m$.

A brief account on Hilbert modular forms can be found in a recent work of the authors \cite[Section~1.2]{2016RAMA} or Trotabas \cite[Section~3]{trotabas}. However, for a more detailed exposition on the topic, the reader is referred to Garrett \cite[Chapter 1, 2]{garrett}, Raghuram-Tanabe \cite[Section 4]{JRMS-2011}, and Shimura \cite[Section 2]{shimura-duke}.

Given two primitive forms $\f\in S_{k}(\ringO_{F})$ and $\g\in S_{k}(\n)$, one defines the $L$-series for the Rankin-Selberg convolution of $\f$ and $\g$ as 
\[L(s,\f\otimes \g)=\zeta_{F}^{\n}(2s)\sum_{\m\subset\ringO_{F}}\frac{C_{\f}(\m)C_{\g}(\m)}{\norm(\m)^{s}},\]
where 
\[\zeta_{F}^{\n}(2s)=\prod_{\p|\n}(1-\norm(\p)^{-2s})\zeta_{F}(2s)=\sum_{d=1}^{\infty}\frac{a_{d}(\n)}{d^{2s}}.\] 
Here, $a_d(\n)$ represents the number of ideals with norm $d$ that are relatively prime to $\n$. 

It follows from the Ramanujan-Petersson bound on the Fourier coefficients of $\f$ and $\g$ (proven by Blasius in \cite{blasius}) that this series converges absolutely for $\Re(s)>1$.   
 We define the Archimedean part of this $L$-function as 
 \[L_{\infty}(s,\f\otimes\g)=\prod_{j=1}^{n}\Gamma\left(s\right)\Gamma\left(s-1+k_{j}\right),\]
 and put
 \[\Lambda(s,\f\otimes \g)=(2\pi)^{-2sn-k}\norm(\d_F^2\n)^{s}L_{\infty}(s,\f\otimes\g)L(s,\f\otimes \g).\]
  Then $\Lambda(s,\f\otimes \g)$ has a meromorphic continuation to $\mathbb{C}$ (see Proposition \ref{prop:shimura} below) and satisfies the functional equation \begin{equation}\label{functional-equation}\Lambda(s,\f\otimes\g)=\Lambda(1-s,\f\otimes \g).\end{equation}
 
\begin{prop}{\rm (Shimura~\cite[Proposition 4.13]{shimura-duke}).} \label{prop:shimura}
The product
\[L(s,\f\otimes\g)L_\infty(s, \f\otimes\g)\]
has a meromorphic continuation to the whole plane, with possible simple poles at $s=1$ and $s=0$. The residue of $L(s,\f\otimes\g)$ at $s=1$ is
\[2^{n-1}(4\pi)^{k}\zeta_{F}(2)\Gamma(k)^{-1}R_F[\ringO_{F}^{\times+}:\ringO_F^{\times 2}]^{-1}\left<\f,\g\right>_\n\]
where $\ringO_F^{\times 2}$ is the group of squares of units in $\ringO_{F}$. 
\end{prop}

\section{Proof of the main theorem}

In this section, we prove Theorem \ref{thm:main}. To this end, we need good lower and upper bounds on the first and second moments, respectively, of the central values $|L(\frac12,\f\otimes\g)|$ for a fixed $\g$ in $\Pi_k(\n) $ as $\f$ varies in $\Pi_k(\ringO_{F})$.
\begin{prop}\label{prop:main}
Let $\g\in\Pi_k(\n)$ be a primitive form. As $k\to\infty$, we have

\[\sum_{\f\in\Pi_k(\ringO_{F})}\left|L\left(\frac12,\f\otimes\g\right)\right|\gg k\]
and

\[\sum_{\f\in\Pi_k(\ringO_{F})}\left|L\left(\frac12,\f\otimes\g\right)\right|^2\ll k\log^{c} k,\]
for some positive integer $c$.
\end{prop}

\begin{proof}
The lower bound on the first moment is proven in Section~\ref{sec:first}. For the upper bound on the second moment, see Section~\ref{sec:second}.
\end{proof}



 
\begin{remark} {\emph{Since the Hilbert modular forms under consideration in this paper are not necessarily of parallel weight, note that we write $k\to \infty$ to mean that $\max\{k_j\}\to \infty$ and $\min\{k_j\}>M$ for a fixed constant $M$.} }
\end{remark}
Assuming Proposition \ref{prop:main}, we complete the proof of the main theorem. By applying Cauchy-Schwarz inequality we obtain
\begin{align*}
k&\ll\sum_{\f\in\Pi_k(\ringO_{F})}\left|L\left(\frac12,\f\otimes \g\right)\right|\ll
\left(\sum_{\substack{\f\in\Pi_k(\ringO_{F}) \\ L(1/2, \f\otimes\g)\neq 0}} 1\right)^{1/2}\left(\sum_{\f\in\Pi_k(\ringO_{F})}\left|L\left(\frac12,\f\otimes\g\right)\right|^2\right)^{1/2}. 
\end{align*}
Hence, we get
\[\sum_{\substack{\f\in\Pi_k(\ringO_{F}) \\ L(1/2, \f\otimes\g)\neq 0}} 1\gg \frac{k}{\log^{c}k}.\] 
 
\vskip .2in


\section{First Moment}\label{sec:first}

The aim of this section is to prove the first part of Proposition~\ref{prop:main}. More precisely, we obtain the following lower bound for the first moment,
\[\sum_{\f\in\Pi_k(\ringO_{F})}\left|L\left(\frac12,\f\otimes\g\right)\right|\gg k.\]
For $\f\in \Pi_{k}(\ringO_{F})$, we define the harmonic weight 
\[ \omega_\f=\frac{\Gamma(k-1)}{(4\pi)^{k-1}|d_F|^{1/2}\left<\f,\f\right>_{\ringO_{F}}},\] 
where $d_F$ is the discriminant of $F$ and $\left<\f,\f\right>_{\ringO_{F}}$ is the Petersson inner product on the space $S_{k}(\ringO_{F})$. 

The point of departure in this section is a twisted first moment of the central values $L(1/2, \f\otimes \g)$ where $\g$ is fixed in $\Pi_{k}(\n)$. More precisely, we consider the weighted harmonic sum
\begin{equation}\label{eqn:average}
\sum_{\f\in\Pi_{k}(\ringO_{F})}L\left(\frac12,\f\otimes \g\right)\omega_\f.
\end{equation}
 A standard application of an approximate functional equation and a Petersson trace formula (see \cite[Proposition~6.3]{trotabas}) allows us to express \eqref{eqn:average} as 

\[\sum_{\f\in\Pi_{k}(\ringO_{F})}L\left(\frac12,\f\otimes \g\right)\omega_\f=M_{\g}(k)+E_{\g}(k)\]
where 
\begin{equation}\label{eqn:main}
M_{\g}(k)=2\sum_{d=1}^{\infty}\frac{a_{d}(\n)}{d}V_{1/2}\left(\frac{4^n\pi^{2n}d^2}{\norm(\d_F^2\n)}\right)
\end{equation}
and
\begin{eqnarray}\label{eqn:error}
E_{\g}(k)&=&2C\sum_{\{\t_\nu\}_\nu}\sum_{\alpha\in(\t_\nu^{-1})^{+}/\ringO_{F}^{\times +}} \frac{C_{\g}(\alpha\t_\nu)}{\sqrt{\norm(\alpha\t_\nu)}}\sum_{d=1}^{\infty}\frac{a_{d}(\n)}{d}V_{1/2}\left(\frac{4^n\pi^{2n}\norm(\alpha\t_\nu)d^2}{\norm(\d_F^2\n)}\right)\\
&&\hspace{.2in}\times\sum_{\substack{\t_\mu^{2}=\t_\nu\\ c\in\t_\mu^{-1}\backslash\{0\}\\ \epsilon\in\ringO_{F}^{\times+}/\ringO_{F}^{\times2}}}\frac{{\mathit Kl}(\epsilon\alpha,\t_\nu; 1,\ringO_F;c,\t_\mu)}{\norm(c\t_\mu)}J_{k-1}\left(\frac{4\pi\sqrt{\epsilon\alpha\left[\t_\nu\t_\mu^{-2}\right]}}{|c|}\right).\nonumber
\end{eqnarray}
Here ${\mathit Kl(*)}$ and $J_{k-1}(*)$ are respectively the Kloosterman sum and a product of the classical $J$-Bessel functions $J_{k_j-1}(*)$, both of which come from the Petersson trace formula. 
The function $V_{1/2}(y)$ originates from the approximate functional equation, and it admits the following integral representation
\begin{equation}\label{eqn:V}
V_{1/2}(y)=\frac{1}{2\pi i}\int_{(3/2)}y^{-u} e^{u^2}\frac{L_\infty(1/2+u,\f\otimes\g)}{L_\infty(1/2,\f\otimes\g)}\frac{du}{u}.
\end{equation}
Moreover, it satisfies 
\begin{equation}\label{V-estimates}V_{1/2}(y)\ll\left(1+\frac{y}{k}\right)^{-A}\quad \text{and}\quad
V_{1/2}(y)=1+O\left(\left(\frac{y}{k}\right)^{\alpha}\right),\end{equation}
where $0<\alpha\leq1$ and the implied constants depend on $A$ and $\alpha$. These estimates follow from Iwaniec-Kowalski \cite[Proposition~5.4]{iwaniec-kowalski}.
We mention here that much of this can be found in a recent work of the authors \cite[Section~2.1]{2016RAMA}. To proceed further, we need the following estimates.

\begin{lemma}\label{lem:first}
As $k$ approaches infinity, we have
\begin{enumerate}
\item $M_{\g}(k)=\gamma_{-1}^{\n}(F)\log(k)+O(1)$, where $\gamma_{-1}^{\n}(F)$ is twice the residue of $\zeta_{F}^{\n}(2u+1)$ at $u=0$, and $\log (k)=\sum_{j=1}^{n}\log(k_{j})$.
 \item $E_{\g}(k)=o(1)$.
\end{enumerate}
\end{lemma}

The proof of this lemma is found in Section~\ref{sec:lem}. Assuming the lemma for now, we have that
\[ \sum_{\f\in\Pi_k(\ringO_{F})}L\left(\frac12,\f\otimes\g\right)\omega_\f=\gamma_{-1}^{\n}\log k+O(1).\]
It follows that
\begin{equation}\label{eqn:lower}
\log k\ll \sum_{\f\in\Pi_k(\ringO_{F})}L\left(\frac12,\f\otimes\g\right)\omega_\f\ll \frac{\log k}{k}\sum_{\f\in\Pi_k(\ringO_{F})}\left|L\left(\frac12,\f\otimes\g\right)\right|.
\end{equation}
Notice that the second inequality in \eqref{eqn:lower} requires the following lemma.

\begin{lemma}\label{lem:omega}
For $\f\in\Pi_{k}(\ringO_{F})$, we have $ \dis \omega_\f\ll \frac{\log k}{ k}$. \end{lemma}
\begin{proof}
Applying Proposition \ref{prop:shimura} gives
\[\omega_{\f}=\frac{2^{n+1}\pi\zeta_F(2) R_F}{(k-1)|d_F|^{1/2}|[\ringO_F^{\times+}:\ringO_F^{\times 2}]\res_{s=1}L(s,\f\otimes\f)}.\]
By assumption, the Dedekind zeta function of $F$ has no Landau-Siegel zero. We may then apply \cite[Theorem~1]{Banks} to deduce that $L(s,\f\otimes\f)=\zeta_{F}(s)L(s,\mathrm{sym}^{2}\f)$ does not admit a Landau-Siegel zero either. Hence, we  get the following lower bound \begin{equation*}\res_{s=1}L(s,\f\otimes\f)\gg (\log k)^{-1},\end{equation*}  thanks to the ground-breaking work of Hoffstein-Lockhart \cite{HL} (namely, Proposition 1.1 therein) and its appendix \cite{GHL}. Using this bound results in the desired upper bound on $\omega_{\f}$.
%
 %
\end{proof}
Thus, we conclude that
\[\sum_{\f\in\Pi_k(\ringO_{F})}\left|L\left(\frac12,\f\otimes\g\right)\right|\gg k,\]
as claimed in Proposition \ref{prop:main}.

\vskip .2in

\section{Proof of Lemma~\ref{lem:first}}\label{sec:lem}

This section is devoted to proving Lemma~\ref{lem:first}. For the first statement, we only sketch a proof since the desired asymptotic formula for $M_{\g}(k)$ is established with an argument very similar to \cite[Section 3]{2016RAMA}. 
Writing $V_{1/2}(y)$ as in (\ref{eqn:V}), we get
\[
M_\g(k)=\frac{1}{2\pi i}\int_{(\frac{3}{2})}\left(\frac{4^n\pi^{2n}}{\norm(\d_F^2\n)}\right)^{-u}e^{u^2}\frac{L_\infty(1/2+u,\f\otimes\g)}{L_\infty(1/2,\f\otimes\g)}\zeta_{F}^{\n}(2u+1)\;\frac{du}{u}.
\]
Since shifting the contour of integration to $\Re(u)=-\frac12+\epsilon$ gives
\begin{align*}\sum_{d=1}^{\infty}\frac{a_{d}(\n)}{d}V_{1/2}\left(\frac{4^n\pi^{2n}d^2}{\norm(\d_F^2\n)}\right)&=\res_{u=0}\left(\left(\frac{4^n\pi^{2n}}{\norm(\d_F^2\n)}\right)^{-u}e^{u^2}\frac{L_\infty(1/2+u,\f\otimes\g)}{L_\infty(1/2,\f\otimes\g)}\frac{\zeta_{F}^{\n}(2u+1)}{u}\right)\\
&\hspace{.4in}+\frac{1}{2\pi i}\int_{(-\frac{1}{2}+\epsilon)}\left(\frac{4^n\pi^{2n}}{\norm(\d_F^2\n)}\right)^{-u}e^{u^2}\frac{L_\infty(1/2+u,\f\otimes\g)}{L_\infty(1/2,\f\otimes\g)}\zeta_{F}^{\n}(2u+1)\;\frac{du}{u},\end{align*}
we need only compute the residue at $u=0$ and bound the integral. Using Stirling's formula, one can estimate that the integral above is $O(k^{-1/2+\epsilon})$, whereas the residue is equal to $\gamma_{-1}^{\n}(F)\log(k)$ plus an explicit constant that depends only on $F$ and $\n$.

Let us now prove the second statement. First, notice that it suffices to consider the partial sum $E_{\g,\nu}(k)$ given by 
\begin{align}\label{eqn:partial-error}
E_{\g,\nu}(k)&=2C\sum_{\alpha\in(\t_\nu^{-1})^{+}/\ringO_{F}^{\times +}} \frac{C_{\g}(\alpha\t_\nu)}{\sqrt{\norm(\alpha\t_\nu)}}\sum_{d=1}^{\infty}\frac{a_{d}(\n)}{d}V_{1/2}\left(\frac{4^n\pi^{2n}\norm(\alpha\t_\nu)d^2}{\norm(\d_F^2\n)}\right)\\ \nonumber 
&\hspace{.6in}\times\sum_{ c\in\t_\mu^{-1}\backslash\{0\}}\frac{{\mathit Kl}(\alpha,\t_\nu; 1,\ringO_F;c,\t_\mu)}{\norm(c\t_\mu)}J_{k-1}\left(\frac{4\pi\sqrt{\alpha\left[\t_\nu\t_\mu^{-2}\right]}}{|c|}\right),
\end{align} 
for any ideal class representative $\t_{\nu}$, while fixing an ideal class representative $\t_{\mu}$ such that $\t_{\mu}^{2}\sim\t_{\nu}$ and ignoring the (finite) sum over $\epsilon$.
 
In the following computations, we use the estimates
 \begin{equation}\label{J-Bessel0}J_{v}(x)\ll \left(\frac{ex}{2(v+1)}\right)^{M-\delta}\quad\quad 0\leq\delta<1\quad\text{and}\quad \delta<M\leq v, \end{equation} 
and
 \begin{equation}\label{J-Bessel00}J_{v}(x)\ll x^{-\frac{1}{2}+\omega}\quad\quad 0\leq\omega<1/2.\end{equation}
Moreover, we need the Weil bound for the Kloosterman sum in the number field setting. This is given by 
\begin{equation}\label{weilbound}|\mathit{Kl}(\alpha,\n;\beta,\m;c,\c)|\ll_{F}\norm\left(((\alpha)\n,(\beta)\m,(c)\c)\right)^{\frac{1}{2}}\tau((c)\c)\norm(c\c)^{\frac{1}{2}},
\end{equation} where $(\a,\b,\c)$ is the gcd of the ideals $\a$, $\b$, $\c$, and $\tau((c)\c)=|\{I\subset \ringO_{F}:(c)\c I^{-1}\subset\ringO_{F}\}|$. 

We rewrite $E_{\g,\nu}(k)$ as
\begin{eqnarray*} E_{\g,\nu}(k)&=&2C\sum_{d=1}^{\infty}\frac{a_{d}(\n)}{d}\sum_{\alpha\in(\t_\nu^{-1})^{+}/\ringO_{F}^{\times +}} \frac{C_{\g}(\alpha\t_\nu)}{\sqrt{\norm(\alpha\t_\nu)}}V_{1/2}\left(\frac{4^n\pi^{2n}\norm(\alpha\t_\nu)d^2}{\norm(\d_F^2\n)}\right)\\
&& \hspace{.2in} \times \sum_{c\in\t_{\mu}^{-1}\backslash\{0\}/\ringO_{F}^{\times +}}\sum_{\eta\in\ringO_{F}^{\times+}}\frac{{\mathit Kl}(\alpha,\t_\nu; 1,\ringO_F;c\eta^{-1},\t_\mu)}{\norm(c\t_\mu)}J_{k-1}\left(\frac{4\pi\eta\sqrt{\alpha\left[\t_\nu\t_\mu^{-2}\right]}}{|c|}\right).
\end{eqnarray*} 
Let $\epsilon>0$ be given. We may truncate the inner sum $\left(\displaystyle{\sum_{\alpha\in(\t_\nu^{-1})^{+}/\ringO_{F}^{\times +}}}\right)$ at $\norm(\alpha\t_{\nu})\ll d^{-2}k^{1+\epsilon}$ with a very small error. To verify this fact, we consider
\begin{eqnarray*} E_{\g,\nu}^{*}(k)&=&2C\sum_{d=1}^{\infty}\frac{a_{d}(\n)}{d}\sum_{\substack{\alpha\in(\t_{\nu}^{-1})^{+}/ \ringO_{F}^{\times +}\\\norm(\alpha\t_{\nu})\gg d^{-2}k^{1+\epsilon}}}  \frac{C_{\g}(\alpha\t_\nu)}{\sqrt{\norm(\alpha\t_\nu)}}V_{1/2}\left(\frac{4^n\pi^{2n}\norm(\alpha\t_\nu)d^2}{\norm(\d_F^2\n)}\right)\\
&& \hspace{.2in} \times \sum_{c\in\t_{\mu}^{-1}\backslash\{0\}/\ringO_{F}^{\times +}}\sum_{\eta\in\ringO_{F}^{\times+}}\frac{{\mathit Kl}(\alpha,\t_\nu; 1,\ringO_F;c\eta^{-1},\t_\mu)}{\norm(c\t_\mu)}J_{k-1}\left(\frac{4\pi\eta\sqrt{\alpha\left[\t_\nu\t_\mu^{-2}\right]}}{|c|}\right).
\end{eqnarray*}
Applying the first estimate in \eqref{V-estimates} and the Weil bound \eqref{weilbound}, together with the well-known estimate $\tau(\m)\ll_\epsilon\norm(\m)^\epsilon$, yields

  \begin{align}\label{eqn:error-tail} 
  E_{\g,\nu}^{*}(k)&\ll\sum_{d=1}^{\infty}\frac{a_{d}(\n)}{d}\sum_{\substack{\alpha\in(\t_{\nu}^{-1})^{+}/ \ringO_{F}^{\times +}\\\norm(\alpha\t_{\nu})\gg d^{-2}k^{1+\epsilon}}} \frac{\left|C_{g}(\alpha\t_{\nu})\right|}{\sqrt{\norm(\alpha\t_{\nu})}}\left(\frac{k}{\norm(\alpha\t_{\nu})d^2}\right)^A\\& \hspace{.2in} \times \left\{\sum_{\substack{c\in\t_{\mu}^{-1}\backslash\{0\}/\ringO_{F}^{\times +}\\|\norm(c\t_{\mu})|\ll\sqrt{\norm(\alpha\t_{\nu})}}}\sum_{\eta\in\ringO_{F}^{\times+}}|\norm(c\t_{\mu})|^{\epsilon-\frac{1}{2}}\prod_{j=1}^n\left|J_{k_j-1}\left(\frac{4\pi\eta_{j}\sqrt{\alpha_{j}\left[\t_\nu\t_\mu^{-2}\right]_{j}}}{|c_{j}|}\right)\right|\right.\nonumber\\
&\qquad \left. {} +\sum_{\substack{c\in\t_{\mu}^{-1}\backslash\{0\}/\ringO_{F}^{\times +}\\|\norm(c\t_{\mu})|\gg\sqrt{\norm(\alpha\t_{\nu})}}}\sum_{\eta\in\ringO_{F}^{\times+}}|\norm(c\t_{\mu})|^{\epsilon-\frac{1}{2}}\prod_{j=1}^n\left|J_{k_j-1}\left(\frac{4\pi\eta_{j}\sqrt{\alpha_{j}\left[\t_\nu\t_\mu^{-2}\right]_{j}}}{|c_{j}|}\right)\right|\right\}\nonumber.\end{align} 
When $|\norm(c\t_{\mu})|\ll\sqrt{\norm(\alpha\t_{\nu})}$, we use the estimate (\ref{J-Bessel00}) with $\omega_j=0$ when $|\eta_{j}|\geq1$ and $\omega_{j}=\omega$ for some fixed $\omega\in(0,1/2)$ otherwise.  We also observe that, for any $a\in F$, there exists a totally positive unit $u$ such that $\norm(a)^{1/n} \ll (au)_j\ll \norm(a)^{1/n}$ for all $j$ (see \cite[Lemma~2.1]{trotabas}). It follows that
\begin{align*}
\prod_{j=1}^nJ_{k_j-1}\left(\frac{4\pi\eta_{j}\sqrt{\alpha_{j}\left[\t_\nu\t_\mu^{-2}\right]_{j}}}{|c_{j}|}\right)&\ll\prod_{j=1}^n\left(\frac{\eta_{j}\sqrt{\alpha_j\left[\t_{\nu}\t_{\mu}^{-2}\right]_j}}{|c_j|}\right)^{-\frac{1}{2}+\omega_{j}}\\ \nonumber
&\ll|\norm(c\t_{\mu})|^{\frac{1}{2}}\sqrt{\norm(\alpha\t_{\nu})}^{-\frac{1}{2}+\omega}\prod_{|\eta_j|<1}|\eta_j|^{\omega}.\end{align*}
When $|\norm(c\t_{\mu})|\gg\sqrt{\norm(\alpha\t_{\nu})}$, we use the estimate (\ref{J-Bessel0})  with $\delta_j$ being chosen as $\delta_j=0$ whenever $|\eta_{j}|\leq1$ and $\delta_{j}=\delta$ for some fixed $\delta\in (0,1)$ otherwise, and $M\leq \min_j\{k_j-1\}$. Once again, we get 
 \begin{align*}
 \prod_{j=1}^nJ_{k_j-1}\left(\frac{4\pi\eta_{j}\sqrt{\alpha_{j}\left[\t_\nu\t_\mu^{-2}\right]_{j}}}{|c_{j}|}\right)&\ll\prod_{j=1}^n\left(\frac{2e\pi\eta_{j}\sqrt{\alpha_j\left[\t_{\nu}\t_{\mu}^{-2}\right]_j}}{k_j |c_j|}\right)^{M}\prod_{j=1}^n\left(\frac{2e\pi\eta_{j}\sqrt{\alpha_j\left[\t_{\nu}\t_{\mu}^{-2}\right]_j}}{k_j |c_j|}\right)^{-\delta_j}\\\nonumber
 &\ll k^{\delta-M}\sqrt{\norm(\alpha\t_{\nu})}^{M}|\norm(c\t_{\mu})|^{\delta-M}\prod_{|\eta_j|>1}|\eta_j|^{-\delta}.\end{align*}
Applying these estimates for the $J$-Bessel functions in \eqref{eqn:error-tail} allows us to factor out the sums over totally positive units as 
$$\sum_{\eta\in\ringO_{F}^{\times+}}\prod_{|\eta_j|<1}|\eta_j|^{\omega}\quad\quad\text{and}\quad\quad \sum_{\eta\in\ringO_{F}^{\times+}}\prod_{|\eta_j|>1}|\eta_j|^{-\delta}.$$ 
These sums are both convergent by virtue of the crucial observation made by Luo in \cite[page~36]{luo1}.
Hence, 
\begin{align*} E_{\g,\nu}^{*}(k)&\ll\sum_{d=1}^{\infty}\frac{a_{d}(\n)}{d}\sum_{\substack{\alpha\in(\t_{\nu}^{-1})^{+}/ \ringO_{F}^{\times +}\\\norm(\alpha\t_{\nu})\gg d^{-2}k^{1+\epsilon}}} \frac{\left|C_{g}(\alpha\t_{\nu})\right|}{\sqrt{\norm(\alpha\t_{\nu})}}\left(\frac{k}{\norm(\alpha\t_{\nu})d^2}\right)^A\\
& \hspace{.2in} \times\left\{ \sum_{\substack{c\in\t_{\mu}^{-1}\backslash\{0\}/\ringO_{F}^{\times +}\\|\norm(c\t_{\mu})|\ll\sqrt{\norm(\alpha\t_{\nu})}}}|\norm(c\t_{\mu})|^{\epsilon} \sqrt{\norm(\alpha\t_{\nu})}^{-\frac{1}{2}+\omega}\right.\\
&\hspace{.6in} \left. {}+k^{\delta-M}\sum_{\substack{c\in\t_{\mu}^{-1}\backslash\{0\}/\ringO_{F}^{\times +}\\|\norm(c\t_{\mu})|\gg\sqrt{\norm(\alpha\t_{\nu})}}}|\norm(c\t_{\mu})|^{\delta-M-\frac12+\epsilon}\sqrt{\norm(\alpha\t_{\nu})}^{M}\right\},\end{align*} 
and therefore,
\begin{align*} E_{\g,\nu}^{*}(k)&\ll k^A\sum_{d=1}^{\infty}\frac{a_{d}(\n)}{d^{1+2A}}\sum_{\substack{\alpha\in(\t_{\nu}^{-1})^{+}/ \ringO_{F}^{\times +}\\\norm(\alpha\t_{\nu})\gg d^{-2}k^{1+\epsilon}}} {\sqrt{\norm(\alpha\t_{\nu})}^{\omega-2A-\frac32+\epsilon}} \sum_{\substack{c\in\t_{\mu}^{-1}\backslash\{0\}/\ringO_{F}^{\times +}\\|\norm(c\t_{\mu})|\ll\sqrt{\norm(\alpha\t_{\nu})}}}|\norm(c\t_{\mu})|^{\epsilon} \\
& \hspace{0.2in}+k^{A+\delta-M} \sum_{d=1}^{\infty}\frac{a_{d}(\n)}{d^{1+2A}}\sum_{\substack{\alpha\in(\t_{\nu}^{-1})^{+}/ \ringO_{F}^{\times +}\\\norm(\alpha\t_{\nu})\gg d^{-2}k^{1+\epsilon}}} {\sqrt{\norm(\alpha\t_{\nu})}^{M-2A-1+\epsilon}}\sum_{\substack{c\in\t_{\mu}^{-1}\backslash\{0\}/\ringO_{F}^{\times +}\\|\norm(c\t_{\mu})|\gg\sqrt{\norm(\alpha\t_{\nu})}}}|\norm(c\t_{\mu})|^{\delta-M-\frac12+\epsilon}\\
&\ll k^{-2018}.
\end{align*}

 
Let us now estimate 
 \begin{eqnarray*} E_{\g,\nu}^{m}(k)&=&2C\sum_{d=1}^{\infty}\frac{a_{d}(\n)}{d}\sum_{\substack{\alpha\in(\t_{\nu}^{-1})^{+}/ \ringO_{F}^{\times +}\\\norm(\alpha\t_{\nu})\ll d^{-2}k^{1+\epsilon}}}  \frac{C_{\g}(\alpha\t_\nu)}{\sqrt{\norm(\alpha\t_\nu)}}V_{1/2}\left(\frac{4^n\pi^{2n}\norm(\alpha\t_\nu)d^2}{\norm(\d_F^2\n)}\right)\\
&& \hspace{.2in} \times \sum_{c\in\t_{\mu}^{-1}\backslash\{0\}/\ringO_{F}^{\times +}}\sum_{\eta\in\ringO_{F}^{\times+}}\frac{{\mathit Kl}(\alpha,\t_\nu; 1,\ringO_F;c\eta^{-1},\t_\mu)}{\norm(c\t_\mu)}J_{k-1}\left(\frac{4\pi\eta\sqrt{\alpha\left[\t_\nu\t_\mu^{-2}\right]}}{|c|}\right).\end{eqnarray*}
To study this sum, we replace the expression $J_{k-1}(*)$ by the product of classical $J$-Bessel functions given by their Mellin-Barnes integral representation. More precisely, we write 
 \[J_{k-1}\left(\frac{4\pi\eta\sqrt{\alpha\left[\t_{\nu}\t_{\mu}^{-2}\right]}}{|c|}\right)=\prod_{j=1}^{n}J_{k_{j}-1}\left(\frac{4\pi\eta_{j}\sqrt{\alpha_{j}\left[\t_{\nu}\t_{\mu}^{-2}\right]_{j}}}{|c_{j}|}\right),\] where for each $j$ we have
\begin{equation}\label{eqn:j-bessel}
J_{k_{j}-1}(x_{j})=\int_{(\sigma_{j})}\frac{\Gamma\left(\frac{k_{j}-1-s_{j}}{2}\right)}{\Gamma\left(\frac{k_{j}-1+s_{j}}{2}+1\right)}\left(\frac{x_{j}}{2}\right)^{s_{j}}\;ds_{j}\quad\quad \text{with}\;\;0<\sigma_{j}<k_{j}-1.
\end{equation}

 We also express the Kloosterman sum explicitly as follows. 
 For $\alpha\in\t_{\nu}^{-1}$ and $c\in\t_{\mu}^{-1}$, the Kloosterman sum $\mathit{Kl}(\alpha,\t_{\nu};1,\ringO_{F};c\eta^{-1},\t_{\mu})$ is given by 

\begin{equation*}\label{eqn:kloosterman}
{\mathit Kl}(\alpha,\t_{\nu}; 1,\ringO_{F};c\eta^{-1},\t_{\mu})=\sum_{x\in \left(\t_{\nu}\d_{F}^{-1}\t_{\mu}^{-1}/\t_{\nu}\d_{F}^{-1}c\right)^{\times}}\exp\left(2\pi i\trace\left(\frac{\alpha x+\left[\t_{\nu}\t_{\mu}^{-2}\right]\overline{x}}{c\eta^{-1}}\right)\right).
\end{equation*}
Here $\overline{x}$ is the unique element in $\left(\t_{\nu}^{-1}\d_{F}\t_{\mu}/\t_{\nu}^{-1}\d_{F}c\t_{\mu}^{2}\right)^{\times}$ such that $x\overline{x}\equiv 1\mod c\t_{\mu}$. The reader is referred to \cite[Section~2.2]{trotabas} for more details on this construction. 

Opening up the $J$-Bessel functions and the Kloosterman sum in $E_{g,\nu}^{m}(k)$  yields
\begin{eqnarray*} E_{\g,\nu}^{m}(k)&=&2C\sum_{d=1}^{\infty}\frac{a_{d}(\n)}{d}\sum_{\substack{\alpha\in(\t_{\nu}^{-1})^{+}/ \ringO_{F}^{\times +}\\\norm(\alpha\t_{\nu})\ll d^{-2}k^{1+\epsilon}}}  \frac{C_{\g}(\alpha\t_\nu)}{\sqrt{\norm(\alpha\t_\nu)}}V_{1/2}\left(\frac{4^n\pi^{2n}\norm(\alpha\t_\nu)d^2}{\norm(\d_F^2\n)}\right) \sum_{c\in\t_{\mu}^{-1}\backslash\{0\}/\ringO_{F}^{\times +}}\frac{1}{\norm(c\t_\mu)}\\
&& \hspace{.2in} \times\sum_{\eta\in\ringO_{F}^{\times+}}\sum_{x\in \left(\t_{\nu}\d_{F}^{-1}\t_{\mu}^{-1}/\t_{\nu}\d_{F}^{-1}c\right)^{\times}}\exp\left(2\pi i\trace\left(\frac{\alpha x+\left[\t_{\nu}\t_{\mu}^{-2}\right]\overline{x}}{c\eta^{-1}}\right)\right)\\
&& \hspace{1.3in} \times\int_{(\sigma)}\frac{\Gamma\left(\frac{k-1-s}{2}\right)}{\Gamma\left(\frac{k-1+s}{2}+1\right)}\left(\frac{2\pi\eta\sqrt{\alpha\left[\t_{\nu}\t_{\mu}^{-2}\right]}}{|c|}\right)^{s}\;ds.\end{eqnarray*}
We note here that multi-index notation is applied again in the integral representation of the $J$-Bessel function. Indeed, 
$\displaystyle{\int_{(\sigma)} \; ds}$ denotes the multiple integration $\displaystyle{\int_{(\sigma_1)}\cdots\int_{(\sigma_n)}}\; ds_n\cdots ds_1$ with  $\sigma_{j}=\sigma+\delta_{j}$ and $\delta_j=\begin{cases}0&\text{if}\;\eta_{j}\geq1\\\delta_{0}&\text{otherwise}\end{cases}$  for some fixed $\sigma>1$ and sufficiently small $\delta_{0}>0$. 
Upon interchanging summations and integration in the expression above, we get 
\begin{eqnarray*} E_{\g,\nu}^{m}(k)&=&2C\sum_{\eta\in\ringO_{F}^{\times+}}\int_{(\sigma)}\frac{\Gamma\left(\frac{k-1-s}{2}\right)}{\Gamma\left(\frac{k-1+s}{2}+1\right)}\left(2\pi\eta\sqrt{\left[\t_{\nu}\t_{\mu}^{-2}\right]}\right)^{s} \sum_{c\in\t_{\mu}^{-1}\backslash\{0\}/\ringO_{F}^{\times +}}\frac{1}{\norm(c\t_\mu)|c|^{s}}\sum_{d=1}^{\infty}\frac{a_{d}(\n)}{d} \\
&& \hspace{.2in}\times\sum_{x\in \left(\t_{\nu}\d_{F}^{-1}\t_{\mu}^{-1}/\t_{\nu}\d_{F}^{-1}c\right)^{\times}} \sum_{\substack{\alpha\in(\t_{\nu}^{-1})^{+}/ \ringO_{F}^{\times +}\\\norm(\alpha\t_{\nu})\ll d^{-2}k^{1+\epsilon}}}  \frac{\alpha^{s/2}C_{\g}(\alpha\t_\nu)}{\sqrt{\norm(\alpha\t_\nu)}}V_{1/2}\left(\frac{4^n\pi^{2n}\norm(\alpha\t_\nu)d^2}{\norm(\d_F^2\n)}\right)\\
&& \hspace{.2in}\times\exp\left(2\pi i\trace\left(\frac{\left[\t_{\nu}\t_{\mu}^{-2}\right]\overline{x}}{c\eta^{-1}}\right)\right)\exp\left(2\pi i\trace\left(\frac{\alpha x}{c\eta^{-1}}\right)\right)\;ds.\end{eqnarray*}

 Let us now analyze the internal sum
\begin{eqnarray*} R_{\g,\nu}(k;\eta,s)&=&\sum_{c\in\t_{\mu}^{-1}\backslash\{0\}/\ringO_{F}^{\times +}}\frac{1}{\norm(c\t_\mu)|c|^{s}}\sum_{d=1}^{\infty}\frac{a_{d}(\n)}{d}\sum_{x} \sum_{\substack{\alpha\in(\t_{\nu}^{-1})^{+}/ \ringO_{F}^{\times +}\\\norm(\alpha\t_{\nu})\ll d^{-2}k^{1+\epsilon}}}  \frac{\alpha^{s/2}C_{\g}(\alpha\t_\nu)}{\sqrt{\norm(\alpha\t_\nu)}}\\
  && \hspace{0.6in}\times V_{1/2}\left(\frac{4^n\pi^{2n}\norm(\alpha\t_\nu)d^2}{\norm(\d_F^2\n)}\right)\exp\left(2\pi i\trace\left(\frac{\left[\t_{\nu}\t_{\mu}^{-2}\right]\overline{x}}{c\eta^{-1}}\right)\right)\exp\left(2\pi i\trace\left(\frac{\alpha x}{c\eta^{-1}}\right)\right)
\end{eqnarray*}
where the sum $\sum_x$ runs over all $x\in \left(\t_{\nu}\d_{F}^{-1}\t_{\mu}^{-1}/\t_{\nu}\d_{F}^{-1}c\right)^{\times}$. 
Applying Cauchy-Schwarz inequality yields
  \begin{eqnarray*}
R_{\g,\nu}(k;\eta,s) &\ll& \sum_{c\in\t_{\mu}^{-1}\backslash\{0\}/\ringO_{F}^{\times +}}\frac{1}{\left|\norm(c\t_\mu)\right||c|^{\Re{s}}}\sum_{d=1}^{\infty}\frac{a_{d}(\n)}{d} \left(\sum_x\left|\exp\left(2\pi i\trace\left(\frac{\left[\t_{\nu}\t_{\mu}^{-2}\right]\overline{x}}{c\eta^{-1}}\right)\right)\right|^2\right)^{1/2} \\
  &&\times \left(\sum_{x}\left|\sum_{\substack{\alpha\in(\t_{\nu}^{-1})^{+}/ \ringO_{F}^{\times +}\\ \norm(\alpha\t_{\nu})\ll d^{-2}k^{1+\epsilon}}}  \frac{\alpha^{s/2}C_{\g}(\alpha\t_\nu)}{\sqrt{\norm(\alpha\t_\nu)}}V_{1/2}\left(\frac{4^n\pi^{2n}\norm(\alpha\t_\nu)d^2}{\norm(\d_F^2\n)}\right)\exp\left(2\pi i\trace\left(\frac{\alpha x}{c\eta^{-1}}\right)\right)\right|^2\right)^{\frac12}.
  \end{eqnarray*}
Next, we employ the additive large sieve inequality (see \cite[page~178]{iwaniec-kowalski})
 \[\sum_{d\in \mathbb{Z}^{n}/c\mathbb{Z}^{n}}\left|\sum_{\substack{v\in\mathbb{Z}^{n}\\v_{j}\leq X}}y_{v}\exp\left(2\pi i\frac{d.v}{c}\right)\right|^2\ll (\norm(c)+X^{n})\sum_{\substack{v\in\mathbb{Z}^n\\v_{j\leq X}}}|y_v|^2.\]
As a result, we get
\begin{eqnarray*}
R_{\g,\nu}(k;\eta,s)&\ll& \sum_{c\in\t_{\mu}^{-1}\backslash\{0\}/\ringO_{F}^{\times +}}\frac{(\left|\norm(c\t_{\mu})\right|+k^{1+\epsilon})^{\frac{1}{2}}}{\left|\norm(c\t_\mu)\right|^{1/2}|c|^{\Re{s}}}\sum_{d=1}^\infty \frac{a_d(\n)}{d} \\
&&\hspace{1in}\times\left(\sum_{\substack{\alpha\in(\t_{\nu}^{-1})^{+}/ \ringO_{F}^{\times +}\\ \norm(\alpha\t_{\nu})\ll d^{-2}k^{1+\epsilon}}}  \left|\frac{\alpha^{s/2}C_{\g}(\alpha\t_\nu)}{\sqrt{\norm(\alpha\t_\nu)}}V_{1/2}\left(\frac{4^n\pi^{2n}\norm(\alpha\t_\nu)d^2}{\norm(\d_F^2\n)}\right)\right|^2\right)^{\frac12}\\
&\ll&\sum_{c\in\t_{\mu}^{-1}\backslash\{0\}/\ringO_{F}^{\times +}}\frac{(\left|\norm(c\t_{\mu})\right|+k^{1+\epsilon})^{\frac{1}{2}}}{\left|\norm(c\t_\mu)\right|^{1/2}|c|^{\Re{s}}}\sum_{d=1}^\infty \frac{a_d(\n)}{d}\\
&&\hspace{1in} \times \left(\sum_{\substack{\alpha\in(\t_{\nu}^{-1})^{+}/ \ringO_{F}^{\times +}\\\norm(\alpha\t_{\nu})\ll d^{-2}k^{1+\epsilon}}}  \frac{\alpha^{\Re{s}}\left|C_{\g}(\alpha\t_\nu)\right|^2}{\norm(\alpha\t_\nu)}\left|V_{1/2}\left(\frac{4^n\pi^{2n}\norm(\alpha\t_\nu)d^2}{\norm(\d_F^2\n)}\right)\right|^2\right)^{\frac12}.
\end{eqnarray*}
This can be further majorized by 
\[\sum_{c\in\t_{\mu}^{-1}\backslash\{0\}/\ringO_{F}^{\times +}}\frac{(\left|\norm(c\t_{\mu})\right|+k^{1+\epsilon})^{\frac{1}{2}}}{\left|\norm(c\t_\mu)\right|^{1/2}|c|^{\Re{s}}}\left(\sum_{\substack{\alpha\in(\t_{\nu}^{-1})^{+}/ \ringO_{F}^{\times +}\\\norm(\alpha\t_{\nu})\ll k^{1+\epsilon}}}  \norm(\alpha\t_{\nu})^{\epsilon-1+\sigma+\delta_0}\right)^{\frac12}\ll k^{(1+\epsilon)\frac{\sigma+\delta_0+1+\epsilon}{2}},\]
%
which allows us to summarize
 \[E_{\g,\nu}^{m}(k)\ll k^{(1+\epsilon)\frac{\sigma+\delta_0+1+\epsilon}{2}}\left|\sum_{\eta\in\ringO_{F}^{\times+}}\int_{(\sigma)}\frac{\Gamma\left(\frac{k-1-s}{2}\right)}{\Gamma\left(\frac{k-1+s}{2}+1\right)}\left(2\pi\eta\sqrt{\left[\t_{\nu}\t_{\mu}^{-2}\right]}\right)^{s}\; ds \right|.\]
 We now write $s_{j}=\sigma+\delta_{j}+it_{j}$. For ease of notation, we put $dt=dt_{1}\cdots dt_{n}$ and $\delta=(\delta_1,\dots,\delta_n)$. Hence, 
\begin{align*}
E_{\g,\nu}^{m}(k)&\ll k^{(1+\epsilon)\frac{\sigma+\delta_0+1+\epsilon}{2}}\sum_{\eta\in\ringO_{F}^{\times+}}\eta^{\sigma+\delta}\int_{-\infty}^{\infty}\left|\frac{\Gamma\left(\frac{k-1-\sigma-\delta-it}{2}\right)}{\Gamma\left(\frac{k-1+\sigma+\delta+it}{2}+1\right)}\right|\;dt
\\&\ll k^{(1+\epsilon)\frac{\sigma+\delta_0+1+\epsilon}{2}}\sum_{\eta\in\ringO_{F}^{\times+}}\prod_{\substack{j=1\\ \eta_{j}<1}}^{n}\eta_{j}^{\delta_0}\int_{-\infty}^{\infty}\left|\frac{\Gamma\left(\frac{k_j-1-\sigma-\delta_0-it_j}{2}\right)}{\Gamma\left(\frac{k_j-1+\sigma+\delta_0+it_j}{2}+1\right)}\right|\;dt_{j}\\
&\hspace{.9in}\times\prod_{\substack{j=1\\\eta_{j}\geq1}}^{n}\int_{-\infty}^{\infty}\left|\frac{\Gamma\left(\frac{k_j-1-\sigma-it_j}{2}\right)}{\Gamma\left(\frac{k_j-1+\sigma+it_j}{2}+1\right)}\right|\;dt_{j}.
\end{align*} 
Using the estimate
\[ \frac{\Gamma(A+c+it)}{\Gamma(A+it)}\ll |A+it|^c\]
for $A>0$ and a real constant $c$ with $|c|<A/2$ (see \cite[Lemma~1]{ganguly-hoffstein-sengupta}), we get 
\begin{align*}
E_{\g,\nu}^{m}(k)&\ll k^{(1+\epsilon)\frac{\sigma+\delta_0+1+\epsilon}{2}}\sum_{\eta\in\ringO_{F}^{\times+}}\prod_{\substack{j=1\\ \eta_{j}<1}}^{n}\eta_{j}^{\delta_0}\int_{-\infty}^{\infty}|k_j+1+\sigma+\delta_{0}+it_{j}|^{-1-\sigma-\delta_0}\;dt_{j}\\
&\hspace{.9in}\times\prod_{\substack{j=1\\\eta_{j}\geq1}}^{n}\int_{-\infty}^{\infty}|k_j+1+\sigma+it_{j}|^{-1-\sigma}\;dt_{j}\\&\ll k^{(1+\epsilon)\frac{\sigma+\delta_0+1+\epsilon}{2}-\sigma}.
\end{align*} 
By choosing $\epsilon$ and $\delta_0$ sufficiently small, we obtain $E_{\g,\nu}^{m}(k)=o(1)$ as $k\to\infty$.

\vskip .2in

\section{Second Moment}\label{sec:second}

The aim of this section is to prove the second statement in Proposition~\ref{prop:main}; namely, 
\[\sum_{\f\in\Pi_k(\ringO_F)}\left|L\left(\frac12,\f\otimes\g\right)\right|^2\ll k(\log k)^{c}.\]
We apply the Rankin-Selberg unfolding method employed by Blomer \cite[pp.~612]{blomer}. To this end, we define the Eisenstein series
\[ E_\nu(z, s)=\sum_{\substack{(c,d)\ringO_F^\times \\ (c,d)\in \t_\nu\n\d_F\times \ringO_F}}\frac{y^s}{|cz+d|^{2s}},\]
for $\nu=1,\dots,h$ and $z\in\h^{n}$. Here, the sum is taken over the $\ringO_F^\times$-equivalence classes of $(c, d)\neq (0, 0)$ where $(c, d)\sim_{\ringO_F^\times}(c',d')$ if $c'=\xi c$ and $d'=\xi d$ for some $\xi\in \ringO_F^\times$. This series is convergent for $\Re(s)>1$. Set \begin{equation}E_\nu^{*}(z, s)=d_{F}^{s}\pi^{-ns}\Gamma^{n}(s)E_\nu(z, s).\end{equation} 
As a function of $s$, the completed Eisenstein series $E^{*}_\nu(z, s)$ can be continued to a meromorphic function on the whole plane with two simple poles at $s=0$ and $s=1$. Standard computations (see for example Freitag \cite[pp.~170]{freitag}) yield the following Fourier expansion 
\begin{align}\label{eqn:eisensteinfourier}
E^{*}_\nu(z, s)&= y^{s}\zeta^{*}_{F,\ringO_{F}}(2s)+y^{1-s}\norm(\t_{\nu}\n\d_{F})^{1-2s}\zeta^{*}_{F,\t_{\nu}\n\d_{F}}(2s-1)\\
&\hspace{.2in} +2^ny^{\frac{1}{2}}\sum_{\substack{\mu\in\t_{\nu}\n/\ringO_{F}^{\times}\\ \mu\neq0}}\norm((\mu)\d_F)^{s-\frac12}\sigma_{1-2s}((\mu)\d_F)\prod_{j=1}^{n}K_{s-\frac12}\left(2\pi|\mu_{j}|y_{j}\right)\exp(2\pi i\trace(\mu x)), \nonumber
\end{align}
where $\zeta^{*}_{F,\ringO_{F}}$ and $\zeta^{*}_{F,\t_{\nu}\n\d_{F}}$ are the completed partial Dedekind zeta functions given by 
\[\zeta^{*}_{F,\a}(s)=d_{F}^{\frac{s}{2}}\pi^{-\frac{ns}{2}}\Gamma^{n}\left(\frac{s}2\right)\norm(\a)^{s}\sum_{\substack{a\in\a/\ringO_{F}\\ a\neq0}}\norm(a)^{-s}.\]
Since $\dis \res_{s=1}\zeta_{F, \ringO_F}(s)=\frac{2^nR_F}{w_F\sqrt{d_F}}$ and $\dis \lim_{s\to 0} s^{1-n} \zeta_{F, \t_\nu\n\d_F}(s)=-\frac{R_F}{w_F}$, we see that the poles at $s=1/2$ from the first and second terms in (\ref{eqn:eisensteinfourier}) cancel out, and therefore we have
\[ E^*_\nu\left(z, \frac12\right)=Cy^{1/2}+2^ny^{1/2}\sum_{\substack{\mu\in \t_\nu\n/\ringO_F^\times \\  \mu\neq 0}}\tau((\mu)\d_F)\prod_{j=1}^n K_0(2\pi|\mu_j|y_j)\exp(2\pi i\trace(\mu x)).\]
As a result, we obtain the estimate
\begin{equation}\label{eqn:eisensteinbound}
E^{*}_\nu\left(z, \frac12\right)\ll y^{\frac{1}{2}}
\end{equation}
for any $z$ in some fundamental domain for $\Gamma_\nu(\n)\backslash\h^n$.

 For a fixed form $\g=(g_1,\dots, g_h)$, we put $\dis v_\nu=v_\nu(z,s)=\norm(\t_\nu)^sg_\nu(z)E_\nu(z, s)$ and $v=(v_1,\dots, v_h)$. Our next goal is to obtain upper and lower bounds for $\| v(z,1/2)\|^2$. 
%
%
Using the definition of inner products given in (\ref{eqn:petersson}), we have 
\[ \|v(z,s)\|^2=\sum_{\nu=1}^h\| v_\nu(z,s)\|^2=\sum_{\nu=1}^{h}\frac{1}{\mu(\Gamma_\nu(\n)\backslash\h^n)}\int_{\Gamma_\nu(\n)\backslash \h^n}|v_\nu(z,s)|^2y^k\, d\mu(z).\]
We note that the measure $\mu(\Gamma_\nu(\n)\backslash \h^n)$ of a fundamental domain for $\Gamma_\nu(\n)\backslash\h^n$ (with respect to $d\mu(z)$) can be written as
\begin{equation}\label{eqn:mu}
 \mu(\Gamma_\nu(\n)\backslash \h^n)=2\pi^{-n} d_F^{3/2}\zeta_F(2)[\ringO_{F}^{\times +}:\ringO_F^{\times 2}]^{-1}\norm(\n)\prod_{\p|\n}(1+\norm(\p)^{-1}).
 \end{equation}
It is obvious from (\ref{eqn:mu}) that $\mu(\Gamma_\nu(\n)\backslash\h^n)$ is independent of $\nu$ which is why we denote it by $M_\n$ hereafter.

Taking $s=\frac12$ in the expression above and using \eqref{eqn:eisensteinbound} yield
\begin{align}\label{eqn:eisenstein}
\left\|v\left(z,\frac12\right)\right\|^2&=M_\n^{-1}\sum_{\nu=1}^{h}|\norm(\t_\nu)|\int_{\Gamma_\nu(\n)\backslash \h^n} |g_\nu(z)|^2 E_\nu\left(z,\frac12\right)^2 y^k d\mu(z)\nonumber  \\
&\ll  M_\n^{-1}\sum_{\nu=1}^{h}|\norm(\t_\nu)| \int_{\Gamma_\nu(\n)\backslash\h^n}|g_\nu(z)|^2y^{1+\frac{1}{\log k}}y^kd\mu(z) \nonumber \\
&\ll  M_\n^{-1}\sum_{\nu=1}^{h}|\norm(\t_\nu)|\int_{\Gamma_\nu(\n)\backslash\h^n}|g_\nu(z)|^2E_\nu\left(z, 1+\frac{1}{\log k}\right)y^kd\mu(z). \end{align}
Upon applying the integral representation of the Rankin-Selberg convolution, we see that \eqref{eqn:eisenstein} can be written as \begin{align*}& (4\pi)^{-\frac{1}{\log k}-k}\d_F^{-1/2}M_\n^{-1}\Gamma\left(k+\frac{1}{\log k}\right)L\left(1+\frac{1}{\log k},\g\otimes\g\right).\end{align*} Note that this follows directly from \cite[Eq~(4.32)~page~670]{shimura-duke} with very minor adjustments to account for the fact that the $L$-function normalization in \cite{shimura-duke} differs from the one used in this paper. Next we invoke \cite[Theorem~2]{CMP}, which provides upper bounds for a general class of $L$-functions at the edge of the critical strip, to get \begin{align}\label{eqn:2_upper}\left\|v\left(z,\frac12\right)\right\|^2&\ll (4\pi)^{-k}\Gamma\left(k+\frac{1}{\log k}\right)(\log k)^{c_1},
\end{align}
 for some explicit constant $c_1$ that depends only on $F$. 

On the other hand, we apply Bessel's inequality and the integral representation of the Rankin-Selberg convolution to obtain
\begin{eqnarray}\label{eqn:bessel}
\|v(z,s)\|^2&\geq &\sum_{\f\in\Pi_k(\ringO_F)}\frac{1}{\|\f\|^2}\left|\left<v(z, s), \f\right>\right|^2\nonumber\\
&=& \sum_{\f\in\Pi_k(\ringO_F)}\frac{1}{\|\f\|^2}\left|(4\pi)^{n(1-s)-k}\d_F^{-1/2}M_\n^{-1}\Gamma(s+k-1)L(s,\f\otimes\g)\right|^2\nonumber\\
&=&(4\pi)^{2n(1-s)-2k}\d_F^{-1}M_\n^{-1}|\Gamma(s+k-1)|^2\sum_{\f\in\Pi_k(\ringO_F)}\frac{1}{\|\f\|^2}|L(s,\f\otimes\g)|^2.
\end{eqnarray}
It follows from \eqref{eqn:bessel}, Proposition \ref{prop:shimura}, and \cite[Theorem~2]{CMP} that
\begin{eqnarray}\label{eqn:2_lower}
\left\|v\left(z,\frac12\right)\right\|^2&\gg& \frac{\Gamma(k-1/2)^2}{(4\pi)^k\Gamma(k)}\sum_{\f\in\Pi_k(\ringO_{F})}\left|L\left(\frac12, \f\otimes\g\right)\right|^2\res_{s=1}L(s,\f\otimes\f)^{-1} \nonumber \\
&\gg& \frac{\Gamma(k-1/2)^2}{(4\pi)^k\Gamma(k)}\sum_\f\left|L\left(\frac12,\f\otimes\g\right)\right|^2(\log k)^{-c_1}.
\end{eqnarray}
Applying \eqref{eqn:2_upper}, \eqref{eqn:2_lower}, and Stirling's formula yields
\begin{eqnarray*}
\sum_{\f\in\Pi_k(\ringO_F)}\left|L\left(\frac12,\f\otimes\g\right)\right|^2&\ll& 
 \frac{(\log k)^{2c_1}\Gamma(k)\Gamma(k+1/\log k)}{\Gamma(k-1/2)^2}\\&\ll& k(\log k)^{2c_1}
\end{eqnarray*}
as desired. 

\section*{Acknowledgements} The authors would like to thank Professors Amir Akbary, Ram Murty, and Nathan Ng for encouraging comments and useful discussions about the topic of this paper.

\bibliographystyle{siam}
\bibliography{references}

\begin{thebibliography}{10}

\bibitem{Banks}
{\sc W.~D. Banks}, {\em Twisted symmetric-square {$L$}-functions and the
  nonexistence of {S}iegel zeros on {${\rm GL}(3)$}}, Duke Math. J., 87 (1997),
  pp.~343--353.

\bibitem{blasius}
{\sc D.~Blasius}, {\em Hilbert modular forms and the {R}amanujan conjecture},
  in Noncommutative geometry and number theory, vol.~E37 of Aspects Math.,
  Wiesbaden, 2006, Vieweg, pp.~35--56.

\bibitem{blomer}
{\sc V.~Blomer}, {\em Period integrals and {R}ankin-{S}elberg {$L$}-functions
  on $\rm{{GL}}(n)$}, Geom. Funct. Anal., 22 (2012), pp.~608--620.

\bibitem{CMP}
{\sc E.~Carletti, G.~Monti~Bragadin, and A.~Perelli}, {\em On general
  {$L$}-functions}, Acta Arith., 66 (1994), pp.~147--179.

\bibitem{chowla}
{\sc S.~Chowla}, {\em The {R}iemann hypothesis and {H}ilbert's tenth problem},
  Mathematics and Its Applications, Vol. 4, Gordon and Breach Science
  Publishers, New York-London-Paris, 1965.

\bibitem{duke}
{\sc W.~Duke}, {\em The critical order of vanishing of automorphic
  {$L$}-functions with large level}, Invent. Math., 119 (1995), pp.~165--174.

\bibitem{freitag}
{\sc E.~Freitag}, {\em Hilbert modular forms}, Springer-Verlag, Berlin, 1990.

\bibitem{ganguly-hoffstein-sengupta}
{\sc S.~Ganguly, J.~Hoffstein, and J.~Sengupta}, {\em Determining modular forms
  on {$SL_2(\mathbb{Z})$} by central values of convolution {$L$}-functions},
  Math. Ann., 345 (2009), pp.~843--857.

\bibitem{garrett}
{\sc P.~Garrett}, {\em Holomorphic {H}ilbert modular forms}, The Wadsworth \&
  Brooks/Cole Mathematics Series, Wadsworth \& Brooks/Cole Advanced Books \&
  Software, Pacific Grove, CA, 1990.

\bibitem{GHL}
{\sc D.~Goldfeld, J.~Hoffstein, and D.~Lieman}, {\em Appendix: An effective
  zero-free region}, Annals of Mathematics, 140 (1994), pp.~177--181.

\bibitem{2016RAMA}
{\sc A.~Hamieh and N.~Tanabe}, {\em Determining {H}ilbert modular forms by
  central values of {R}ankin-{S}elberg convolutions: {t}he {W}eight {A}spect},
  Ramanujan J., 45 (2018), pp.~615--637.

\bibitem{HL}
{\sc J.~Hoffstein and P.~Lockhart}, {\em Coefficients of {M}aass forms and the
  {S}iegel zero. {W}ith an appendix by {D}orian {G}oldfeld, {H}offstein and
  {D}aniel {L}ieman}, Ann. of Math. (2), 140 (1994), pp.~161--181.

\bibitem{iwaniec-kowalski}
{\sc H.~Iwaniec and E.~Kowalski}, {\em Analytic Number Theory}, vol.~53 of
  American Mathematical Society Colloquium Publications, American Mathematical
  Society, Providence, 2004.

\bibitem{lau-tsang}
{\sc Y.-K. Lau and K.-M. Tsang}, {\em A mean square formula for central values
  of twisted automorphic {$L$}-functions}, Acta Arith., 118 (2005),
  pp.~231--262.

\bibitem{liu-masri}
{\sc S.-C. Liu and R.~Masri}, {\em Nonvanishing of {R}ankin-{S}elberg
  {$L$}-functions for {H}ilbert modular forms}, Ramanujan J., 34 (2014),
  pp.~227--236.

\bibitem{luo1}
{\sc W.~Luo}, {\em Poincar\'e series and {H}ilbert modular forms}, Ramanujan
  J., 7 (2003), pp.~129--140.

\bibitem{mv}
{\sc P.~Michel and A.~Venkatesh}, {\em The subconvexity problem for {$\GL_2$}},
  Publ. Math. IHES, 111 (2010), pp.~171--271.

\bibitem{JRMS-2011}
{\sc A.~Raghuram and N.~Tanabe}, {\em Notes on the arithmetic of {H}ilbert
  modular forms}, J. Ramanujan Math. Soc., 26 (2011), pp.~261--319.

\bibitem{shimura-duke}
{\sc G.~Shimura}, {\em The special values of the zeta functions associated with
  {H}ilbert modular forms}, Duke Math. J., 45 (1978), pp.~637--679.

\bibitem{sound}
{\sc K.~Soundararajan}, {\em Nonvanishing of quadratic {D}irichlet
  {$L$}-functions at {$s=\frac12$}}, Ann. of Math. (2), 152 (2000),
  pp.~447--488.

\bibitem{stark}
{\sc H.~M. Stark}, {\em Some effective cases of the {B}rauer-{S}iegel theorem},
  Invent. Math., 23 (1974), pp.~135--152.

\bibitem{trotabas}
{\sc D.~Trotabas}, {\em Non annulation des fonctions {$L$} des formes
  modulaires de {H}ilbert au point central}, Ann. Inst. Fourier, Grenoble, 61
  (2011), pp.~187--259.

\end{thebibliography}

\end{document}